\documentclass[11pt]{article}
\usepackage{amssymb}
\usepackage{latexsym}
\usepackage{amsthm}
\usepackage{amscd}
\usepackage{amsmath}
\usepackage{diagrams}
\usepackage{indentfirst}

\theoremstyle{definition}
\newtheorem{thm}{Theorem}[section]
\newtheorem{lem}[thm]{Lemma}
\newtheorem{cor}[thm]{Corollary}
\newtheorem{defn-lem}[thm]{Definition-Lemma}

\newtheorem{prop}[thm]{Proposition}

\newtheorem{rem}[thm]{Remark}

\newtheorem{defn}[thm]{Definition}
\numberwithin{equation}{section}

\allowdisplaybreaks[4]

\def \Q{{\mathbb Q}}
\def \N{{\mathbb N}}
\def \C{{\mathbb C}}

\def \Z{{\mathbb Z}}

\def \P{\mathbb P}

\def\map#1.#2.{#1 \longrightarrow #2}
\def\rmap#1.#2.{#1 \dasharrow #2}
\DeclareMathOperator{\rank}{rank}

\DeclareMathOperator{\Pic}{Pic}
\DeclareMathOperator{\Spec}{Spec}

\def\fb#1.{\underset #1 \to \times}
\def\pr#1.{\Bbb P^{#1}}
\def\ring#1.{\mathcal O_{#1}}
\def\mlist#1.#2.{{#1}_1,{#1}_2,\dots,{#1}_{#2}}

\def\uloopr#1{\ar@'{@+{[0,0]+(-4,5)} @+{[0,0]+(0,10)}
@+{[0,0]+(4,5)}}
  ^{#1}}
\def\dloopr#1{\ar@'{@+{[0,0]+(-4,-5)} @+{[0,0]+(0,-10)}
@+{[0,0]+(4,-5)}}
  _{#1}}

\def\rloopd#1{\ar@'{@+{[0,0]+(5,4)} @+{[0,0]+(10,0)}
@+{[0,0]+(5,-4)}}
  ^{#1}}
\def\lloopd#1{\ar@'{@+{[0,0]+(-5,4)} @+{[0,0]+(-10,0)}
@+{[0,0]+(-5,-4)}}
  _{#1}}

\long\def\ignore#1{}
\long\def\ignore#1{#1}
\title{Generic ordinarity for Semi-Stable Fibrations}
\author{Junmyeong Jang}
\date{email : jang3@math.purdue.edu}

\begin{document}

\maketitle
\begin{center}
\end{center}

\medskip
     \section{Introduction}
     Let $k$ be a field and $\pi : X \to C$ be a semi-stable fibration of a smooth proper surface over $k$
     to a proper smooth curve over $k$. If $k$ is a subfield of the field of complex numbers, $\C$,
     the following semi-positivity theorem holds.
     \\ \\
     $\bf{Theorem.}$(Semi-Positivity Theorem, Xiao) If $\pi :X \to C$ is a fibration of a proper smooth surface to a proper smooth
     curve over $\C$, then all the quotient bundles of $\pi _{*}\omega _{X/C}$ are of non-negative degree.
     \cite{X},p.1\\ \\
     In other words, all the Harder-Narasimhan slopes of $\pi _{*} \omega _{X/C}$ are non-negative, or equivalently
     all the Harder-Narasimhan slopes of $R^{1} \pi _{*} \mathcal{O}_{X}$ are non-positive.
     But over a positive characteristic filed, the semi-positivity theorem does not hold in general.
     Moret-Bailly constructed a non-isotrivial semi-stable fibration of fiber genus 2, $\pi _{M} : X_{M} \to \P ^{1}$
     such that
     $$R^{1} \pi _{M*} \mathcal{O}_{X_{M}} = \mathcal{O}(-p) \oplus \mathcal{O}(1).\textrm{\cite{MB},p.137}$$
     Here $p$ is the characteristic of the base field.
     $X_{M}$ is a
     theta divisor of a principal polarization of a non-isotrivial supersingular abelian surface over $\P ^{1}$.
     Hence every special fiber of $\pi _{M}$ is either a supersingular smooth curve of genus 2 or a union of two supersingular
     elliptic curves which intersect at a point transversally. We can see the $p$-rank of the generic fiber of
     $\pi _{M}$ is 0.
     The main work of this paper is to prove the semi-positivity theorem for a semi-stable fibration over
     a field of positive characteristic provided with the $p$-rank of the generic fiber is maximal
     or equivalently the fibration is generically ordinary.
     Precisely,
     \newpage
     $\bf{Theorem \ 1.}$
     If $\pi: X \to C$ is generically ordinary semi-stable fibration, then
     \begin{itemize}
     \item[(a)] $\dim H^{0}(R^{1} \pi _{*} \mathcal{O} _{X}) = \dim H^{0} (R^{1} \pi^{p^{n}} _{*} \mathcal{O}_{X^{p^{n}}})$ and $ \dim H^{1}(\mathcal{O}_{X})= \dim H^{1}(\mathcal{O}_{X^{p^{n}}})$ for any $n$.
     \item[(b)] All the Harder-Narasimhan slopes of $R^{1}\pi _{*}(\mathcal{O}_{X})$ are
     non-positive.
     \end{itemize}
     Here $\pi ^{p^{n}} : X^{p^{n}} \to C$ is the base change of $\pi : X \to C$ by the
     $n$-iterative Frobenius
     morphism of $C$, $F^{n}_{C} : C \to C$.
     The proof of the theorem is given in section 2.
     Let me summarize the idea of the proof of the theorem.
     When
     $F_{X/C} : X \to X^{p}$ is the relative Frobenius morphism,
     $F^{*}_{X/C} : \mathcal{O}_{X^{p}} \to
     F_{X/C} ^{*} \mathcal{O}_{X}$ is injective and the cokernel of this morphism is
     denoted by $B^{1} \omega _{X/C}$ or $B^{1} \omega$. $B^{1}\omega$ is flat over
     $\mathcal{O}_{C}$. If $\pi$ is generically ordinary,
     $H^{0}(B^{1} \omega | X_{s})=H^{1}(B^{1} \omega | X_{s})=0$ for almost all $s \in
     C$, hence $\pi ^{p} _{*} B^{1} \omega=0$
     and $R^{1} \pi ^{p}_{*} B^{1} \omega$ is torsion because $\pi$
     is relatively 1-dimensional. Considering the exact sequence of coherent $\mathcal{O}_{C}$-modules,
     $$ \to \pi ^{p} _{*} B^{1} \omega \to R^{1} \pi _{*} \mathcal{O} _{X} \to F_{C}^{*}
     R^{1} \pi _{*} \mathcal{O} _{X} \to R^{1} \pi ^{p} _{*} B^{1} \omega \to 0$$
     $F_{C}^{*} R^{1} \pi _{*} \mathcal{O}_{X}$ is a subbundle of $R^{1} \pi _{*} \mathcal{O}_{X}$, so
     $\dim H^{0}(R^{1} \pi _{*} \mathcal{O} _{X}) \geq \dim H^{0} (F_{C} ^{*} R^{1} \pi _{*} \mathcal{O}_{X})$.
     But also we have
     $\dim H^{0}( R^{1} \pi _{*} \mathcal{O} _{X}) \leq \dim H^{0} (F_{C} ^{*} R^{1} \pi _{*} \mathcal{O}_{X})$.
     Hence
     $\dim H^{0}( R^{1} \pi _{*} \mathcal{O} _{X}) = \dim H^{0} (F_{C} ^{*} R^{1} \pi _{*} \mathcal{O}_{X})$.
     And by the Leray spectral sequnece,
     $ \dim H^{1}(\mathcal{O}_{X})= \dim H^{1}(\mathcal{O}_{X^{p}})$.
     Repeating this argument we can show
     $$\dim H^{0}(R^{1}\pi _{*} \mathcal{O}_{X} = \dim H^{0}(R^{1} \pi ^{p} _{*} \mathcal{O}_{X^{p^{n}}})
     \textrm{ and } \dim H^{1}(\mathcal{O}_{X}) = \dim H^{1}(\mathcal{O}_{X^{p^{n}}}).$$
     The part (b) follows the part (a) and the Rieman-Roch theorem.
     This is a somewhat interesting phenomenon in a sense that it relates the ordinarity, a Galois theoretic condition,
     to a numerical property of slopes of vector bundles.

     In section 3, as an application of of the main theorem,
     we will construct a counterexample of Parshin's expectation on the Miyaoka-Yau inequality.
     Parshin thought that the failure of the Miyaoka-Yau inequality is related to the non-smoothness of
     the Picard scheme, so he conjectured that a version of the Miyaoka-Yau inequality holds if the Picard scheme
     of a given surface of general type is smooth. \cite{PA},p.288 We will construct a smooth
     proper surface of general type over a finite field whose
     Picard scheme is smooth and $c_{1} ^{2} > M c_{2}$ for any $M >0$.

     $\bf{Acknowledgments}$

     I deeply appreciate to Prof.M.Kim.
     Without his sincere support and instruction, this work would have been impossible.
     It's pleasure to thank Prof.K.Joshi and Prof.L.Illusie
     for many helpful advice and encouragements. I also thanks to Prof.F.Oort, Prof.C.Chai, Prof.D.Arapura and
     Dr.D-U.Lee for answers to questions, comments and conversations.
     \newpage
     \section{Definitions and the Proof of the Main Theorem}
     Let $k$ be an algebraically closed field and $C$ be a projective curve over $k$.
     \begin{defn}
     $C$ is (semi-)stable if
     \begin{itemize}
     \item[1.] It is connected and reduced.
     \item[2.] All the singular points are normal crossing.
     \item[3.] An irreducible component, which is isomorphic to $\P^{1}$, meets other components
     in at least 3(resp. 2) points.
     \end{itemize}
     \end{defn}
     For an arbitrary base scheme, we define a (semi-)stable curve as follows.
     \begin{defn}
     A proper flat morphism of relative dimension 1 of schemes
     \\ $\pi : X \to S$ is a (semi-)stable curve
     if
     every geometric fiber of $\pi$ is
     a (semi-)stable curve in the sense of definition 2.7.
     \end{defn}
     In this paper, we are mainly concerned with generically smooth semi-stable curves
     over a proper smooth curve over a field. If $\pi : X \to C$ is such a semi-stable
     fibration over an algebraically closed field $k$, $X$ is a proper surface over $k$ and
     the singular points of $X$ are isolated. A singularity of $X$ is \'{e}tale locally isomorphic to
     $$k[x,y,t]/(t^{n}-xy).$$
     If  $\tilde{X} \to X$ is the minimal blow up of these singularities, the composition
     $\tilde{\pi} : \tilde{X} \to C$ is also a semi-stable fibration.\cite{DE},p.4 Moreover $\omega ^{1} _{\tilde{\pi}/C}$
     is isomorphic to the pull back of $\omega ^{1} _{X/C}$ by the blow up and
     $\tilde{\pi} _{*} \omega ^{1} _{\tilde{X}/C}
     = \pi _{*} \omega ^{1} _{X/C}$.\cite{SZ1},p.171
     Hence for many purpose, we may assume that $X$ is a smooth surface over $k$.
     From now on, we assume $X$ is a smooth surface over $k$ and $\pi : X \to C$ is a
     generically smooth semi-stable fibration unless it is stated otherwise.
     \begin{defn}
     A semi-stable fibration $\pi : X \to C$ is isotrivial, if all the special fibers of $\pi$ are isomorphic.
     \end{defn}
     In particular, an isotrivial fibration $\pi : X \to C$ is a smooth fibration. If $\pi$
     is isotrivial, there exists a finite \'{e}tale cover $C' \to C$ such that
     the base change $\pi _{C'} : X \times_{C} C' \to C'$
     is trivial. In particular, if $\pi$ is isotrivial, $\deg \pi _{*} \omega_{X/C}=0$.
     \begin{prop}
     (Szpiro) If $\pi$ is a non-isotrivial semi-stable fibration, \mbox{$\deg \pi _{*} \omega _{X/C} > 0$}.
     Equivalently, $\deg R^{1} \pi _{*} \mathcal{O}_{X} < 0$.\cite{SZ1},p.173
     \end{prop}

     Now assume that $k$ is a perfect field of positive characteristic $p$ and that $X$ is a smooth proper variety defined over $k$.
     We have the Frobenius
     diagram for $X/k$,
     $$\begin{diagram}
     X & \rTo ^{F_{X/k}} & X^{p} &
     \rTo & X \\
     & \rdTo & \dTo & & \dTo \\
     & & k & \rTo ^{F_{k}} & k.
     \end{diagram}$$
     Here $F _{X/k}$ is the relative Frobenius morphism of $X/k$.
     When $\Omega ^{\cdot} _{X/k}$ is the DeRham complex of $X/k$,
     $F_{X/k *}(\Omega^{\cdot}_{X/k})$ is an $\mathcal{O}
     _{X^{p}}$-linear complex of coherent $\mathcal{O}_{X^{p}}$-modules. The image of $F_{X/k *}
     \Omega^{i-1} _{X/k} \to F_{X/k *} \Omega^{i} _{X/k}$ is
     denoted by $B^{i} \Omega _{X/k}$ or $B^{i}\Omega$.  Each $B^{i}\Omega$ is a vector
     bundle on $X^{p}$.
     \defn $X$ is ordinary (Bloch-Kato ordinary) if $H^{i}(B^{j}\Omega_{X/k})=0$ for all
     $i$ and $j$. \\

     There are many equivalent conditions to Bloch-Kato
     ordinarity\cite{IR},p.209.
     If $X$ is a curve or an abelian
     variety, $X$ is ordinary if and only if it satisfies the classical definition,
     that the
     order of $p$-torsion of the $\Pic _{X/k}^{0}$ is maximal or
     that the Frobenius morphism on $H^{1}(\mathcal{O}_{X})$ is bijective.
     If all the integral crystalline cohomologies of $X$, $H^{i}_{cris}(X/W)$
     are torsion free, $X$ is ordinary if and only if
     the Newton polygons of $X$ are equal to the Hodge polygons of $X$ for
     all degrees. Here $W$ is the ring of Witt vectors of $k$.

     We can extend the definition of ordinarity to any proper smooth morphism
     of schemes of characteristic $p$. Assume $f: X\to S$ be a proper and
     smooth morphism.
     And let $X^{p}= X \times _{S} (S,F_{S})$ and $F_{X/S} :
     X \to X^{p}$ be the relative Frobenius morphism. The image $B^{i}_{X/S}$ of $F_{X/S *}
     \Omega ^{i-1}_{X/S} \overset{d}{\to} F_{X/S *}\Omega ^{i} _{X/S}$ is a vector bundle on $X^{p}$. We define
     $X/S$ to be ordinary if $R^{i}f_{*} (B^{j}_{X/S})=0$ for all $i$ and $j$.

     Moreover, the notion of ordinarity can be extended to
     a proper generically smooth morphism to a $\Spec$ of a discrete valuation ring
     with normal crossing on the special fiber. We recall the definition of ordinarity
     for such a morphism from \cite{I2} and \cite{I3}.
     Let $A$ be a discrete valuation ring
     of positive characteristic and $S = \Spec A$. $s \in S$ is the closed point.

     \defn
     $f:X \to S$ is locally semi-stable if it is isomorphic to
     $$\Spec A[x_{1}, \cdots ,
     x_{n}]/(x_{1} \cdots x_{r} -t) \to \Spec A$$
      \'{e}tale locally at a relative singular point, where $t$ is a unifomizer of
     $A$. \\ \\
     The term ``locally semi-stable morphism" is not conventional.
     Usually such a morphism is just called semi-stable.
     Here we introduce this definition
     to avoid a conflict with the former definition of semi-stable curve.
     Note that the definition of semi-stable curve is little different from
     that of a locally semi-stable morphism. In the definition of semi-stable curve, it is required that the semi-stable
     fibration is proper and relatively minimal while in definition 2.6, the morphism should
     be generically smooth.
     However if $X \to
     C$ is a generically smooth relative semi-stable curve over a smooth curve, $X
     \otimes \mathcal{O}_{C,s} \to \Spec \mathcal{O}_{C,s}$ is locally semi-stable
     as per definition 2.6 for each $s \in C$.

     Let $X \to S$ be a locally semi-stable morphism and $U \subset X$ be
     the relative smooth locus and $u: U \hookrightarrow X$
     be the inclusion. Then $X \setminus U$ is of codimension at least 2,
     hence $\omega^{\cdot} _{X/S} = u_{*} \Omega^{\cdot} _{U/S}$ is a complex of locally
     free sheaves on $X$ and $\omega ^{i}_{X/S} = \wedge ^{i}
     \omega ^{1} _{X/S}$. When $X$ is given as $A[x_{1}, \cdots ,
     x_{n}]/( x_{1} \cdots x_{r}-t)$ \'{e}tale locally, $\omega ^{1}
     _{X/S}$ is the free module of rank $n-1$, generated by
     $$dx_{1} / x_{1},
     \cdots , dx_{r} /x_{r} , dx_{r+1}, \cdots , dx_{n}$$ with the
     relation $$ \sum _{i=1} ^{r} dx_{i} / x_{i} =0.$$
     Note that if
     the relative dimension of $f$ is $d$, the highest wedge product $\omega
     ^{d} _{X/S}$ is the relative dualizing sheaf of $f:X \to S$.

     Let $X^{p}$ be the base change of $X$ by the Frobenius
     morphism of $S$ and \\ $F_{X/S} : X \to X^{p}$ be the relative
     Frobenius morphism. Then $F_{X/S *} \omega ^{\cdot} _{X/S}$ is an
     $\mathcal{O} _{X^{p}}$-linear complex. The image and the kernel of the differentials of the complex
     $F_{*}\omega ^{\cdot}_{X/S}$ are denoted by $B^{i}\omega _{X/S},$ and $Z^{i}\omega _{X/S}$ respectively,
     and $\mathcal{H} ^{i} \omega _{X/S} = Z^{i} \omega _{X/S} / B^{i} \omega _{X/S}$.
     $B^{i}\omega _{X/S}, Z^{i} \omega _{X/S}, \mathcal{H} ^{i} \omega _{X/S}$
     are $\mathcal{O}_{X^{p}}$-coherent sheaves
     and flat over $S$.
     The usual Cartier
     isomorphism
     $$\textrm{C}^{-1} : \Omega^{i} _{U^{p} /S} \to \mathcal{H}
     ^{i} F_{*} \Omega ^{\cdot} _{U/S}$$
     on the smooth locus extends to an isomorphism\cite{I2}p.381
     $$\textrm{C} ^{-1} : \omega ^{i} _{X^{p}/S} \to \mathcal{H} ^{i} F_{*}
     \omega ^{\cdot} _{X/S}.$$
     Here $\omega ^{i} _{X^{p}/S}=F_{C}^{*}(\omega ^{i} _{X/S})$.
      In particular, the Cartier
     isomorphism at $i=0$ gives an exact sequence
     $$ 0 \to \mathcal{O}_{X^{p}} \to F_{*} \mathcal{O}_{X} \to
     B^{1} \omega _{X/S} \to 0.$$
     \begin{defn}
     A proper locally semi-stable morphism $f:X \to S$ is ordinary if
     $H^{j}(B^{i} \omega _{X/S})=0$ for all $i,j$.
     \end{defn}
     Since $B^{i} \omega _{X/S}$ are flat over $S$, $f$ is ordinary if and only if
     $H^{j}(X_{s}, B^{i} \omega
     _{X/S}|X_{s})=0$ for all $i,j$ when $X_{s}$ is the special fiber. This
     definition depends on the entire $X \to S$, and not only on the
     special fiber. But if $f$ is smooth, $f$ is ordinary if and only if the special fiber is ordinary.
     Moreover if the relative dimension of $f$ is 1 and
     the residue field is perfect, $f$ is ordinary if and only if
     the Frobenius morphism on $H^{1}(\mathcal{O}_{X_{s}})$ is
     bijective, hence the ordinarity depends only on the special fiber.

     Since all the above arguments are local on the base $S$, they are
     still valid if we replace the base $S$ by a smooth curve over
     a perfect field of positive characteristic.
     Let $C$ be a smooth curve over a perfect field $k$ of positive characteristic and
     $\pi :X \to C$ be a proper generically smooth semi-stable curve.
     Since each $B^{i} \omega _{X/C}$ is flat over $C$,
     by the semi-continuity theorem, the set of points $s \in C$ satisfying the property that
     $X \otimes _{\mathcal{O}_{C}}
     \mathcal{O}_{s}$
     is ordinary, forms an
     open set in $C$.
     \begin{defn}
     Let $\pi :X \to C$ be a proper semi-stable curve. We say that $\pi$ is generically ordinary
     if at least one closed fiber of $\pi$ is ordinary. (Hence almost all closed
     fibers of $\pi$ are ordinary.)
     \end{defn}
     Now we recall the (semi-)stability and the Harder-Narasimhan
     slopes of a vector bundle on a smooth proper curve.
     Let $C$ be a smooth proper curve defined over an
     algebraically closed field $k$.
     For a vector bundle $V$ on $C$, the slope of $V$
     is defined as $s(V)= \deg V/ \rank V$. $V$ is called
     semi-stable (resp. stable) if for any proper subbundle $W$ of $V$, it is
     satisfied that $s(W) \leq s(V)$. (resp.
     $s(W) < s(V)$)
     \begin{prop}
     (Harder-Narasimhan)For any vector bundle $V$ on $C$, there exists a
     unique filtration of $V$ consisting of subbundles of $V$,
     $$0 = V_{0} \subset V_{1} \subset \cdots \subset V_{n} =V$$
     such that $V_{i}/V_{i-1}$ is a semi-stable vector bundle of slope $\lambda
     _{i}$ and $\lambda _{1} > \lambda _{2} > \cdots > \lambda
     _{n}$./cite{HN},p.220
     \end{prop}

     This filtration is called the Harder-Narasimhan
     filtration of $V$ and $\lambda _{1}, \cdots , \lambda _{n}$ are
     called the Harder-Narasimhan slopes of $V$.
     If the base field is not algebraically closed, the Harder-Narasimhan slopes of $V$
     are defined as the Harder-Narasimhan slopes of pullback of the bundle along the
     base change to an algebraically closed field.
     When $\pi : X \to C$ is a semi-stable fibration of
     a proper smooth surface to a proper smooth curve over a
     subfield of $\C$, the semi-positivity theorem states that all the Harder-Narasimhan slope
     of $R^{1}\pi _{*} \mathcal{O}_{X}$ is non-positive. \cite{X},p.1
      \\ \\
     $\bf{Theorem \ 1.}$
     If $\pi: X \to C$ is generically ordinary semi-stable fibration, then
     \begin{itemize}
     \item[(a)] $\dim H^{0}(R^{1} \pi _{*} \mathcal{O} _{X}) = \dim H^{0} (R^{1} \pi ^{p^{n}} _{*} \mathcal{O}_{X^{p^{n}}})$ and $ \dim H^{1}(\mathcal{O}_{X})= \dim H^{1}(\mathcal{O}_{X^{p^{n}}})$ for any $n$.
     \item[(b)] All the Harder-Narasimhan slopes of $R^{1}\pi _{*}(\mathcal{O}_{X})$ are
     non-positive.
     \end{itemize}
      \begin{proof}
     Let $X^{p}$ be the base change of $X$ by the absolute Frobenius
     morphism of $C$. There is the Frobenius diagram of $X/C$,
     $$(2.1) \ \begin{diagram}
     X & \rTo ^{F_{X/C}} & X^{p} &
     \rTo & X \\
      & \rdTo _{\pi} & \dTo > {\pi ^{p}}& & \dTo >{\pi} \\
      & & C & \rTo ^{F_{C}} & C.\\
      \end{diagram}$$
      We have an exact sequence of coherent $\mathcal{O} _{X ^{p}}$-modules
      $$(2.2) \ 0 \to \mathcal{O}_{X^{p}} \to F_{X/C*}
      \mathcal{O}_{X} \to B^{1}\omega \to 0.$$
       The long exact sequence via $\pi ^{p}_{*}$ for $(2.2)$ is
      $$(2.3) \ 0 \to \mathcal{O}_{C} \simeq \mathcal{O}_{C} \to \pi ^{p}
      _{*} (B^{1}\omega) \to R^{1}\pi ^{p} _{*} (\mathcal{O} _{X^{p}}) \overset{F_{X/C}^{*}}{\to}
      R^{1} \pi _{*} (\mathcal{O}_{X}) \to R^{1} \pi^{p} _{*} (B^{1}\omega) \to 0.$$
      Since $\pi$ is generically ordinary, the restriction of $\pi
      ^{p} _{*} B^{1}\omega$ to the ordinary locus in $C$ is 0. But
      $B^{1}\omega$ is flat over $\mathcal{O}_{C}$, so $\pi ^{p} _{*} B^{1}\omega =0$.
      Hence in $(2.3)$, $F^{*}$ is injective and
      $$\dim H^{0}(R^{1}\pi ^{p} _{*} (\mathcal{O}_{X^{p}})) \leq \dim H^{0}(R^{1} \pi _{*} (\mathcal{O} _{X})).$$
      On the other hand, because
      $$H^{0}(R^{1} \pi^{p}_{*} (\mathcal{O}_{X^{p}})) =
      H^{0}(F_{C}^{*} R^{1} \pi _{*} \mathcal{O}_{X}) =
      H^{0}(R^{1} \pi _{*} (O_{X}) \otimes _{\mathcal{O}_{C}}
      F_{C*}(\mathcal{O}_{C}))$$
      and there is an injection $R^{1}
      \pi _{*} \mathcal{O}_{X} \hookrightarrow R^{1} \pi_{*} \mathcal{O}_{X}
      \otimes F_{C*}(\mathcal{O}_{C})$,
      $$\dim H^{0}(R^{1} \pi_{*} \mathcal{O}_{X}) \leq \dim \ H^{0}
      (F_{C}^{*} R^{1} \pi_{*} \mathcal{O}_{X}).$$
      Therefore
      $$\dim H^{0}(R^{1} \pi_{*} \mathcal{O}_{X}) = \dim \ H^{0}
      (F_{C}^{*} R^{1} \pi_{*} \mathcal{O}_{X}).$$
      Since
      $$\dim  H^{1}(\mathcal{O}_{X}) =
      \dim H^{1}(\mathcal{O}_{C}) +
      \dim H^{0}(R^{1} \pi_{*} \mathcal{O}_{X})$$
      and
      $$\dim H^{1}(\mathcal{O}_{X^{p}}) =
      \dim H^{1}(\mathcal{O}_{C}) +
      \dim H^{0}(R^{1} \pi^{p}_{*} \mathcal{O}_{X^{p}}),$$
      we have
      $$\dim  H^{1}(\mathcal{O}_{X})=\dim H^{1}(\mathcal{O}_{X^{p}}).$$
      We can apply this argument to
      the relative Frobenius morphism $F_{X^{p^{i}}/C} : X ^{p^{i}} \to X^{p^{i+1}}$ for any $i$,
      since $F_{X^{p^{i}}/C} : X^{p^{i}} \to X^{p^{i+1}}$ is the base change of the
      relative Frobenius morphism $F_{X/C} : X \to X^{p}$ by $F^{i} _{C} : C \to C$.
      Then by the induction, we have
      $$\dim H^{0}(R^{1} \pi_{*} \mathcal{O}_{X}) = \dim \ H^{0}
      (F_{C}^{n*} R^{1} \pi_{*} \mathcal{O}_{X})$$
      and
      $$\dim  H^{1}(\mathcal{O}_{X})=\dim H^{1}(\mathcal{O}_{X^{p^{n}}})$$
      for any $n$. This proves (a).
      Now assume $V$ is a subbundle of $R^{1} \pi _{*} \mathcal{O}_{X}$ with a positive degree $d>0$ and rank $r$.
      Then $F_{C} ^{n*} V$ is a subbundle of $F_{C}^{n*} R^{1} \pi _{*} \mathcal{O}_{X}$ with a positive
      degree $p ^{n}d$ and rank $r$. By the Rieman-Roch theorem on $C$,
      $\dim H^{0}(F_{C} ^{n*} V)$ diverges as $n$ goes to infinity.
      But this is a contradiction to the fact that $\dim H^{0}(F_{C} ^{n*} R^{1} \pi _{*} \mathcal{O}_{X})$
      is stable. Hence $R^{1} \pi _{*} \mathcal{O}_{X}$ does not have
      a subbundle of a positive degree, and all the Harder-Narasimhan slopes of $R^{1} \pi _{*} \mathcal{O}_{X}$
      are non-positive.
      \end{proof}
      \subsection{Triviality of slope 0 part of $R^{1}\pi _{*}\mathcal{O}_{X}$}
      In the proof of Theorem 1,  we actually proved that if $\pi$ is generically ordinary,
      $F_{C} ^{n*} R^{1} \pi _{*} \mathcal{O}_{X}$
      does not have a positive Harder-Narasimhan slope for any $n \in \N$. Hence
      the slope 0 part of $R^{1} \pi _{*} \mathcal{O}_{X}$ is strongly semi-stable. In fact, we can say
      more about the slope 0 part of $R^{1} \pi _{*} \mathcal{O}_{X}$. For convenience, we will use
      the following notations.
      \begin{defn}
      For a vector bundle $V$ on $C$, $V _{0}$ is the slope 0 part of $V$ and $V_{-}$ is the negative
      slope part of $V$.
      \end{defn}
      If $\pi$ is generically ordinary, there is the canonical filtration of $R^{1} \pi _{*} \mathcal{O}_{X}$,
      $$ 0 \to (R^{1} \pi _{*} \mathcal{O}_{X})_{0} \to R^{1} \pi _{*} \mathcal{O}_{X}
      \to (R^{1} \pi _{*} \mathcal{O}_{X})_{-} \to 0.$$
      \begin{defn}
      A vector bundle $V$ on $C$ is potentially trivial if there exists a finite \'{e}tale cover
      $f: D \to C$ such that $f ^{*} V$ is trivial.
      \end{defn}

      \begin{prop}
      If $\pi : X \to C$ is generically ordinary,
      $(R^{1} \pi
      _{*} \mathcal{O}_{X})_{0}$ is potentially trivial.
      \end{prop}
      \begin{proof}
      In the proof of Theorem 1, we saw that
      $$ F_{C}^{*} R^{1} \pi _{*} \mathcal{O}_{X} \hookrightarrow
      R^{1} \pi _{*} \mathcal{O}_{X}.$$
      On the other hand, in the canonical exact sequence,
      $$ 0 \to (R^{1} \pi _{*} \mathcal{O}_{X})_{0}  \to R^{1} \pi _{*} \mathcal{O}_{X} \to
      (R^{1} \pi _{*} \mathcal{O}_{X})_{-} \to 0,$$
      $(R^{1} \pi _{*} \mathcal{O}_{X}) _{-}$ is an iterative extension of
      semi-stable vector bundles of negative slopes. Therefore the image of the
      composition
      $$F^{*} (R^{1} \pi _{*} \mathcal{O}_{X})_{0} \hookrightarrow F^{*} R^{1} \pi _{*}
      \mathcal{O}_{X} \hookrightarrow R^{1} \pi _{*}
      \mathcal{O}_{X}$$
      is contained in $(R^{1} \pi _{*} \mathcal{O}_{X})_{0}$. Since $(R^{1} \pi _{*} \mathcal{O}_{X})_{0}$
      and $F^{*} (R^{1} \pi _{*} \mathcal{O}_{X})_{0}$ are of the same rank and
      the same degree, they should be isomorphic. Now it's enough to show the
      following lemma. The author learned this lemma from K.Joshi. The
      original source of this fact is apparently a letter from Mumford to Seshadri.
      \begin{lem}
      If $M$ is a vector bundle on $C$ such that $F^{*}M
      \simeq M$ then $M$ is potentially trivial.
      \end{lem}
      \begin{proof}
      Let $r$ be the rank of $M$. $M$ is a class in
      $H^{1}_{et}(GL_{r}(\mathcal{O}_{C}))$. If we can show that $M$
      is actually in $H^{1}_{et}(GL_{r}(\mathbb{F}_{p}))$, it corresponds
      to a group homomorphism $\pi _{1}(C) \to GL_{r}(\mathbb{F}
      _{p})$. Since $GL_{r}(\mathbb{F} _{p})$ is a finite group, we
      have a finite \'{e}tale Galois cover $f:D \to C$, corresponding
      to the kernel of this group homomorphism. Then $f^{*} M$ is the trivial vector
      bundle of rank r on $D$.

      We proceed to prove $M$ is a class in $H^{1}(GL_{r}(\mathbb{F} _{p}))$.
      Let $\{(U_{i}, g_{ij})\}$ be a trivialization of $M$ and the
      transition functions. Then $\{(U_{i} ,
      g_{ij}^{(p)})\}$ is a trivialization of $F^{*}M$ and transition functions.
      Here $g_{ij}^{(p)}$ is the matrix obtained by taking the $p$-power of each entry of $g_{ij}$.
      From the
      assumption, after replacing $U_{i}$ by a refinement if
      necessary, there are $h_{i} \in GL_{r}(\mathcal{O}_{C})(U_{i})$, such that
      $$h_{i} g^{(p)}_{ij} h_{j}^{-1} = g_{ij}.$$
      To show $M \in H^{1}(GL_{r}(\mathbb{F} _{p}))$, we should
      find $k_{i} \in GL_{r}(\mathcal{O}_{C}) (U_{i})$ such that
      $$(k_{i}g_{ij}k_{j}^{-1})^{(p)} = k_{i}^{(p)} g_{ij}^{(p)}
      (k_{j}^{-1})^{(p)} = k_{i}g_{ij} k_{j}^{-1}.$$
      This is equivalent to
      $$k_{i}^{-1}k_{i}^{(p)} g_{ij}^{(p)} (k_{j}^{-1})^{(p)} k_{j} =
      g_{ij}.$$
      Hence the problem is reduced to find $k_{i}$ satisfying $k_{i}^{-1}
      k_{i}^{(p)} = h_{i}$. Since it is local in the \'{e}tale topology,
      it is enough to find $k_{i}$ in the strict henselization for a geometric point
      of $C$. Using Hensel's lemma, it is again reduced to the
      same problem over an algebraically closed field of characteristic
      $p$. But by Lang's Theorem\cite{SP},p.76,
      $$ A \mapsto A^{-1}A^{(p)}$$
      is surjective on $GL_{r}(\bar{k})$. Hence there exists $k_{i}$ satisfying
      $k_{i} ^{-1} k_{i} ^{(p)}= h_{i}$.
      This proves the claim.
      \end{proof}
      \end{proof}
      \begin{rem}
      It is natural to expect that if $\pi : X \to C$ is defined over a field of characteristic 0,
      $(R^{1} \pi _{*} \mathcal{O}_{X})_{0}$ is potentially trivial.
      But for this problem, we can't apply the standard reduction argument directly.
      One reason
      is that in the reduction situation,
      we don't know whether there are infinitely many places at which the reduction is generically ordinary.
      This obstruction is related to Serre's ordinary reduction conjecture.
      \end{rem}

      \section{Counterexample to Parshin's conjecture}
      In this section, we will construct a counterexample to Parhin's conjecture.
           Let us recall the construction of a counterexample to the Miyaoka-Yau inequality
      over a field of positive characteristic
      from \cite{SZ1}. Let $k$ be a perfect field of positive characteristic.
      $\pi : X \to C$ is a smooth non-isotrivial fibration of a proper smooth surface
      to a proper smooth curve over $k$ of fiber genus $g \geq 2$ and of base genus $q \geq 2$.
      Also set $d = - \deg R^{1} \pi _{*} \mathcal{O}_{X} > 0$. Then
      $$c_{1}^{2}(X) = 12d + 8(q-1)(g-1)\textrm{ and }c_{2}(X)= 4(q-1)(g-1).$$
      When $\pi ^{p^{n}} : X^{p^{n}} \to C$ is the base change of $\pi$ by the $n$-iterative
      Frobenius morphism of $C$, $F^{n}_{C} : C \to C$,
      $\deg R^{1} \pi ^{p^{n}} _{*} \mathcal{O} _{X ^{p^{n}}} = -p^{n}d$ and
      $$c_{1}^{2}(X^{p^{n}}) = 12dp^{n} + 8(q-1)(g-1)\textrm{ and }c_{2}(X)= 4(q-1)(g-1).$$
      For any $M>0$, if $n$ is sufficiently large,
      $$c_{1} ^{2} (X^{p^{n}}) > M c_{2} (X^{p^{n}}).$$
      \begin{lem}
      Suppose that $X$ is a smooth proper surface over $k$ which admits a smooth fibration,
      $\pi : X \to C$, to a smooth proper curve $C$ over $k$. If $\pi$ is
      generically ordinary and $\Pic X$ is smooth, then
      $\Pic X^{p^{n}}$ is smooth for any $n \in  \N$ when $X^{p^{n}} \to C$ is the
      base change of $X \to C$ by $n$-iterative Frobenius morphism $F_{C}^{n} : C \to C$.
      \end{lem}
      \begin{proof}
      Recall the Frobenius diagram
      $$\begin{diagram}
     X & \rTo ^{F_{X/C}} & X^{p} &
     \rTo ^{\alpha} & X \\
     & \rdTo & \dTo & & \dTo \\
     & & C & \rTo ^{F_{C}} & C.
     \end{diagram}$$
      Here $\alpha \circ F_{X/C}$ is the absolute Frobenius morphism
      of $X$ and $F_{X/C} \circ \alpha$ is the absolute Frobenius morphism of
      $X^{p}$. Since $\pi$ is smooth, $X^{p}$ is smooth over $k$. For a smooth proper variety,
      the Frobenius morphism induces a bijective
      semi-linear morphism on the rational crystalline cohomologies. Therefore
      $$ \dim H^{i}_{cris} (X/K)=\dim H^{i}_{cris} (X^{p}/K).$$
      Here $K$
      is the fraction field of the ring of Witt vectors $W=W(k)$
      and $H^{i}_{cris}(X/K)=H^{i}_{cris}(X/W) \otimes K$.  In
      particular,
      $$\dim H^{1}_{cris} (X/K)=\dim H^{1}_{cris} (X^{p}/K).$$
      On the
      other hand, the $K$-dimension of the crystalline cohomology $H^{i}_{cris} (X/K)$ is
      equal to the $\Q _{l}$-dimension of the $l$-adic \'{e}tale
      cohomology $H^{i}_{et} (\bar{X},\Q _{l})$, where $l$ is a prime number
      different from the characteristic of $k$ and $\bar{X} = X \times _{k} \bar{k}$.
      The dimension of $H^{1}_{et} (\bar{X},\Q _{l})$ is the twice
      of the dimension of $\Pic X$, so
      $$\dim \Pic X = \dim \Pic X^{p}.$$
      $\Pic X$ is smooth if and only
      if the dimension of $\Pic X$ is equal to the $k$-dimension of
      $H^{1}(\mathcal{O}_{X})$. Since $\pi$ is generically
      ordinary, by Thm 1.(a),
      $$\dim H^{1}(\mathcal{O}_{X})=\dim H^{1}(\mathcal{O}_{X^{p}}).$$
      Hence if $\Pic X$ is smooth, $\Pic X^{p}$ is smooth.
      By the same argument, $\Pic X^{p^{n}}$ is smooth for any $n$.
      \end{proof}
      \begin{cor}
      For any $M>0$, there is a smooth proper surface of general type $X$ over a
      finite field whose Picard scheme is smooth and
      $c_{1}^{2}(X) > M c_{2}(X)$.
      \end{cor}
      \begin{proof}
      By the above lemma, it is enough to give a non-isotrivial generically ordinary smooth
      fibration $X \to C$ such that $\Pic (X)$ is smooth. We will construct
      such an example by a reduction argument.

      Let $F_{m}$ be the Fermat curve $x^{m} + y^{m} +z^{m}=0$
      over $\C$, with $m>3$. Denote the
      genus of $F_{m}$ by $g$ and notice $g \geq 3$. Let $\mathcal{M}_{g}$ be the moduli space of
      smooth proper curves of genus $g$ over $\C$. By \cite{HA},p.105,
      there exists a smooth proper curve $C_{0}$ in
      $\mathcal{M} _{g}$ passing through the point representing
      $F_{m}$. Then
      there is a finite cover $C \to C_{0}$ and a non-isotrivial
      smooth fibration $\pi : X \to C$ which induces the composition $C
      \to C_{0} \hookrightarrow \mathcal{M} _{g}$. Let us choose $s \in C$
      such that $X_{s} = X \times _{C} k(s)= F_{m}$.
      We can take an integral model of $\pi$ with the section $s$ over a noetherian
      domain of finite type over $\Z$.
      Explicitly, we can take $A$, an integral domain of finite type over
      $\Z$,
      and a smooth fibration $\pi _{A}: X_{A} \to C_{A}$
      over $\Spec A$ satisfying
      \begin{enumerate}
      \item $X_{A}$ and $C_{A}$ are smooth and proper over
      $\Spec A$.
      \item There is a geometric generic point of $\eta : \C \to \Spec A$
      such that $\pi _{A} \times  _{A} \eta$ is isomorphic to $\pi : X \to C$.
      \item There exists a section $S:\Spec A \to C_{A}$ such that
      $S \times _{A} \eta$ corresponds to $s$ with respect to the isomorphism in 2.
      \item $S \times _{C_{A}} X_{A}$ is isomorphic to
      the Fermat curve over $\Spec A$.
      \end{enumerate}
      Since $\Spec A$ is a scheme of finite type over $\Z$, there is a
      rational point of $\Spec A$ over a number field $F$. Considering
      the coordinates of this rational point, there is a
      morphism $\Spec B \to \Spec A$, where $B$ is a localization
      of $ \mathcal{O}_{F}$, the ring of integers of $F$.
      By the bases change, we obtain a
      smooth fibration $\pi _{B} : X _{B} \to
      C _{B}$ over $\Spec B$. Then for a place $\upsilon
      \in \Spec B$, the fiber of $\pi _{\upsilon}= \pi _{B} \times k_{\upsilon} $ over
      $S_{\upsilon}$ is the Fermat curve
      over the residue field of $\upsilon$.
      The ordinarity of the Fermat curve over
      a finite field depends only on the characteristic of the
      field. To be precise, it is ordinary if and only if
      $p \equiv
      1 \textrm{ mod }m$ where $p$ is the characteristic of the finite field.\cite{Y}
      Hence at infinitely many places of $\Spec B$,
      the reduction of $\pi$ is generically ordinary.
      Because $\Pic X_{\upsilon}$ is smooth for almost all $\upsilon \in \Spec B$,
      there is a place $\upsilon \in \Spec B$ such that $\pi _{\upsilon}$ is
      generically ordinary and the Picard scheme of $X
      _{\upsilon}$ is smooth.
      \end{proof}
      \begin{rem}
      In the above example, while $\textrm{dim }H^{1}(\mathcal{O}_{X^{p^{n}}})$ is stable
      by the Frobenius base change, $\textrm{dim }H^{0}(\Omega ^{1} _{X^{p^{n}}})$ is strictly increasing\cite{FO},p.94. Considering the inequality
      $$ c_{1}^{2} \leq 5c_{2} + 6\beta _{1} + 6(2h^{1,0}-\beta _{1}),$$
      it seems that the Miyaoka-Yau inequality is related to
      the ``correctness'' of $h^{1,0}$ rather than that of $h^{0,1}$. The example
      we have constructed shows that the ``correct value'' of $h^{0,1}$ does not guarantee the Miyaoka-Yau inequality.
      \end{rem}

     \end{document}